\documentclass[a4paper, reqno]{amsart}
\usepackage[dvipsnames]{xcolor}
\usepackage{amsmath,amsthm,amssymb,relsize}
\usepackage[nobysame,alphabetic]{amsrefs}
\usepackage[margin=30mm]{geometry}

\usepackage{booktabs, multirow, adjustbox, array}

\usepackage{pifont}
\usepackage{mathrsfs}
\usepackage{enumerate}
\usepackage{tikz}
\usepackage{epigraph}
\usepackage{enumitem}

\usetikzlibrary{
  cd,
  calc,
  positioning,
  fit,
  arrows,
  decorations.pathreplacing,
  decorations.markings,
  shapes.geometric,
  backgrounds,
  bending
}
\usepackage{tikzsymbols}
\usepackage[arrow,matrix,curve,cmtip,ps]{xy}
\usepackage{pb-diagram}
\usepackage{pb-xy}

\definecolor{darkblue}{rgb}{0,0,0.6}

\usepackage[breaklinks, pdftex, ocgcolorlinks,colorlinks=true, citecolor=darkblue, filecolor=darkblue, linkcolor=darkblue, urlcolor=darkblue]{hyperref}
\usepackage[capitalize,noabbrev]{cleveref}

\numberwithin{equation}{section}
\newtheorem{thm}{Theorem}[section]
\newtheorem{lem}[thm]{Lemma}

\newtheorem{conj}[thm]{Conjecture}
\newtheorem{prop}[thm]{Proposition}

\theoremstyle{definition}
\newtheorem{defn}{Definition}[section]
\newtheorem*{ack}{Acknowledgments}
\newtheorem*{con}{Conventions}
\newtheorem*{org}{Organisation}

\theoremstyle{remark}

\newtheorem*{rem}{Remark}

\newcommand{\Q}{\mathbb{Q}}

\newcommand{\R}{\mathbb{R}}
\newcommand{\Z}{\mathbb{Z}}

\newcommand{\im}{\operatorname{Im}}

\DeclareMathOperator{\Aut}{Aut}
\DeclareMathOperator{\bAut}{bAut}

\DeclareMathOperator{\Wh}{Wh}

\DeclareMathOperator{\BSTOP}{BSTOP}
\DeclareMathOperator{\BTOPSpin}{BTOPSpin}
\DeclareMathOperator{\TOPSpin}{TOPSpin}

\DeclareMathOperator{\coker}{coker}

\begin{document}
\title{Unknotting nonorientable surfaces of genus 4 and 5}

\author{Mark Pencovitch}
\address{School of Mathematics and Statistics, University of Glasgow, United Kingdom}
\email{m.pencovitch.1@research.gla.ac.uk}

\def\subjclassname{\textup{2020} Mathematics Subject Classification}
\expandafter\let\csname subjclassname@1991\endcsname=\subjclassname

\begin{abstract}

Expanding on work by Conway, Orson, and Powell, we study the isotopy classes rel.\ boundary of nonorientable, compact, locally flatly embedded surfaces in \(D^4\) with knot group ~\(\Z_2\). In particular we show that if two such surfaces have fixed knot boundary \(K\) in ~\(S^4\) such that ~\(\vert \det(K) \vert =1\), the same normal Euler number, and the same nonorientable genus \(4\) or ~\(5\), then they are ambiently isotopic rel.\ boundary.

This implies that closed, nonorientable, locally flatly embedded surfaces in the 4-sphere with knot group \(\Z_2\) of nonorientable genus 4 and 5 are topologically unknotted. The proof relies on calculations, implemented in Sage, which imply that the modified surgery obstruction is elementary. Furthermore we show that this method fails for nonorientable genus \(6\) and \(7\).
\end{abstract}
\maketitle

\section{Introduction}


We are interested in studying surfaces embedded in \(4\)-manifolds. We can think of such surfaces as the one dimension higher analogue of knots embedded in \(3\)-manifolds. Indeed, the questions we want to answer for both objects are the same. When is a surface unknotted? Can we classify them in a meaningful way?

In this paper, we answer these questions for a small class of nonorientable surfaces.

\begin{defn}
    A \textit{\(\Z_2\)-surface in \(S^4\)} is a closed, connected, nonorientable surface \(F\) locally flatly embedded in \(S^4\) with knot group \(\pi_1(S^4\backslash F) \cong \Z_2\).

    A \textit{\(\Z_2\)-surface in \(D^4\)} is a properly, locally flatly embedded, compact, nonorientable surface \(Z\) with boundary a knot \(K\) in \(S^3\) and knot group \(\pi_1(D^4 \backslash Z) \cong \Z_2\). 

    The \textit{nonorientable genus} of a compact, connected, nonorientable surface \(F\) with at most one boundary component is given by \(\dim(H_1(F;\Z_2))\).

    The \textit{normal Euler number} of a \(\Z_2\)-surface \(F\) is the Euler number associated to the normal bundle of \(F\). If \(\partial F = K\), then this is the Euler number relative to the \(0\)-framing of \(K\).
\end{defn}

Let \(h\) be the nonorientable genus of a closed surface \(F\) locally flatly embedded in \(S^4\). It is a conjecture by Whitney, proved by Massey \cite{Massey}, that the normal Euler number \(e(F)\) is given by one of \(\{-2h, -2h+4, \dots, 2h-4, 2h \}\). Note that for a closed, connected, \textit{orientable} surface, locally flatly embedded in \(S^4\), the normal Euler number is always zero. With Massey's theorem in mind, We say that a nonorientable surface \(F\) locally flatly embedded in \(S^4\) with nonorientable genus \(h\) has \textit{extremal} normal Euler number if \(e(F) = \pm 2h\). It is shown in \cite[Lemma \(5.2\)]{og} that Massey's theorem can be reformulated in the case of non-empty boundary. A consequence of this is that we say a nonorientable surface \(F\) locally flatly embedded in \(D^4\) with nonorientable genus \(h\) has \textit{extremal} normal Euler number if \(\vert e(F)- 2\sigma(K) \vert = 2h\), where \(\sigma(K)\) denotes the signature of the boundary knot.

The following is our main theorem.

\begin{thm}\label{main}
    Let \(F_0,F_1\) be \(\Z_2\)-surfaces in \(D^4\) of nonorientable genus \(h\), boundary \(K\), and the same extremal normal Euler number. Assume that \(\vert \det(K) \vert =1\). If \(h = 4,5\) then \(F_0\) and \(F_1\) are ambiently isotopic rel.\ boundary.

    In particular, every \(\Z_2\)-surface in \(S^4\) with extremal normal Euler number and nonorientable genus \(4\) or \(5\) is topologically unknotted.
\end{thm}

We say a \(\Z_2\)-surface in \(S^4\) is topologically \textit{unknotted} if it is isotopic to the connect sum of copies of \(\R P^2\) standardly embedded in \(S^4\). See \cite{Law} for a precise construction.

In a recent paper \cite[Theorem A]{og}, Conway-Orson-Powell proved the above results when ~\(F_0\) and \(F_1\) have non-extremal normal Euler number, or when \(h\leq 3\) in the extremal case. This, combined with our above result, is significant progress towards the following conjecture.

\begin{conj}\label{conj}
    Every \(\Z_2\)-surface in \(S^4\) is topologically unknotted.
\end{conj}

We know that if a closed, connected, nonorientable surface \(F\) locally flatly embedded in \(S^4\) is topologically unknotted, then it has knot group \(\pi_1(S^4\backslash F) \cong \Z_2\). However we are yet to find a ~\(\Z_2\)-surface in \(S^4\) which is topologically knotted. Lawson \cite{Law2} proved the conjecture in the case of \(h=1\) and Kreck \cite{Kreck2} gave a proof for the case of \(h=10\) and normal Euler number \(16\).

This paper answers this conjecture in the cases of extremal normal Euler number and nonorientable genus \(4\) and \(5\). We follow the same method of proof as in \cite{og}, and obtain our results by using Sage to adapt and streamline the calculations presented in their appendix. 

Furthermore, we show that this method fails when \(h=6,7\), meaning that we cannot conclude anything by this strategy for these cases. We conjecture that the method also fails for \(h\geq 8\).

\begin{con} 
    As in \cite{og}, we work in the topological category with locally flat embeddings. We assume that all manifolds are compact, connected and based unless stated otherwise. If a manifold has nonempty boundary, then the basepoint is assumed to lie in the boundary.
\end{con}

\begin{org}
    In Section~\ref{sec:back} we give a brief background and outline of the proof given in \cite{og}. In Section~\ref{sec:calc} we go on to demonstrate this method by expanding on their calculations. Finally, in Section~\ref{sec:comm} we discuss possible generalisations and further research. 
\end{org}

\begin{ack}
 The author would like to thank Anthony Conway, Patrick Orson, and Mark Powell for their invaluable guidance, feedback, and patience.
\end{ack}

\section{Background}\label{sec:back}

Throughout this section, unless stated otherwise, assume that \(F_0\) and \(F_1\) are \(\Z_2\)-surfaces in \(D^4\) with nonorientable genus \(h\), extremal normal Euler number \(e\), and \(\partial F_i =K\) a fixed knot such that ~\(\vert \det(K) \vert =1\).

\subsection{The modified surgery obstruction}

The results in \cite{og} use Kreck's modified surgery with respect to the normal \(1\)-type of the exteriors of the \(\Z_2\)-surfaces. See \cite{Kreck, CrowSixt} for more details on the theory of this modified surgery. 

Let \(M_i\) denote the exterior of the \(\Z_2\)-surface \(F_i\) in \(D^4\) for \(i\in \{0,1\}\). These are spin \(4\)-manifolds with fundamental group \(\Z_2\). A choice of spin structure on \(M_i\) along with an identification ~\(\pi_1(M_i) \cong ~\Z_2\) determines a \(2\)-connected map \(\Bar{\nu}_i\colon M_i \to B\) such that \(B:= \BTOPSpin \times B\Z_2\). There exists a \(2\)-coconnected fibration \(\xi\colon B \to \BSTOP\) and we say that the pair \((B, \xi)\) is the \textit{normal \(1\)-type} of \(M_i\). The map \(\Bar{\nu}_i\) is called a \textit{normal \(1\)-smoothing} if it lifts the stable normal bundle ~\(M_i \to \BSTOP\) of \(M_i\).

Suppose \(M_0\) is homeomorphic to \(M_1\), via a homeomorphism that extends to a homeomorphism ~\(D^4\to D^4\) sending \(F_0\) to \(F_1\) and which restricts to the identity on \(S^3\). By the Alexander trick, we know that homeomorphisms of \(D^4\) that restrict to the identity on \(S^3\) are rel.\ boundary isotopic to the identity relative to the boundary. This implies that \(F_0\) and \(F_1\) are ambiently isotopic rel.\ boundary.

We can use Kreck's modified surgery to show that \(M_0\) and \(M_1\) are homeomorphic. However we need some extra conditions on the homeomorphism to ensure that it has the above properties and guarantees that the surfaces are ambiently isotopic rel.\ boundary.

\begin{defn}
    Denote by \(S(\nu F_i) = \partial\Bar{\nu}F_i \backslash \nu K \subseteq D^4\), the boundary of the closed tubular neighbourhood of the \(F_i\) minus a neighbourhood of the knot \(K\). A homeomorphism \(f\colon S(\nu F_0) \to S(\nu F_1)\) is called \textit{\(\nu\)-extendable} if it extends to a homeomorphism \(\Bar{\nu}F_0 \cong \Bar{\nu}F_1\) which sends \(F_0\) to \(F_1\). If \(f\) restricts to the identity \(S(\nu K) \to S(\nu K)\), then \(f\) is \textit{\(\nu\)-extendable rel.\ boundary} if \(f\) is \(\nu\)-extendable and the homeomorphism \(\Bar{\nu}F_0 \cong \Bar{\nu}F_1\) also restricts to the identity on the closed neighbourhood \(\Bar{\nu}K\).
\end{defn}

If \(f\) is \(\nu\)-extendable rel.\ boundary, then it also extends to a homeomorphism \(\partial M_0 \to \partial M_1\) which restricts to the identity on the knot exterior. For \(F_0\) and \(F_1\) to be ambiently isotopic rel.\ boundary, there must exist a homeomorphism \(S(\nu F_0) \to S(\nu F_1)\) which is \(\nu\)-extendable rel.\ boundary.

In \cite[Proposition \(6.3\)]{og} it was shown that if two \(\Z_2\)-surfaces in \(D^4\) have the same boundary knot, nonorientable genus, and normal Euler number, then there exists an orientation preserving homeomorphism \(f\colon \partial M_0 \to \partial M_1\) such that \(f\) restricts to a homeomorphism \(S(\nu F_0) \cong S(\nu F_1)\) which is \(\nu\)-extendable rel.\ boundary and the closed \(4\)-manifold \(M:= M_0 \cup_f -M_1\) is spin and has ~\(\pi_1(M) \cong \Z_2\).

We can use this \(M\) to construct a \((B, \xi)\)-cobordism rel.\ boundary between \(M_0\) and \(M_1\), which we can attempt to surger using modified surgery. In particular, in \cite[Proposition \(8.3\)]{og} it was shown that \((M_0,\Bar{\nu}_0)\) and \((M_0, \Bar{\nu}_0)\) are \((B,\xi)\)-cobordant using the above fact by showing that \(M\) has zero signature and using the Atiyah-Hirzebruch spectral sequence to calculate \[\Omega_4(B,\xi) \cong \Omega_4^{\TOPSpin} (B) \cong \Z.\] This implies that \(M\) is null-bordant over \((B,\xi)\) and thus we get a cobordism \((W, \Bar{\nu}_W)\) between the two surface exteriors.

Then, using the modified surgery technique in \cite[Theorem \(4\)]{Kreck}, since \(\Wh(\Z_2)=0\), we can try to surger the cobordism  \((W,\Bar{\nu}_W)\) to an \(s\)-cobordism. By the \(s\)-cobordism theorem  \cite[Theorem \(7.1\)A]{FQ}, and using the fact that the group \(\Z_2\) is \textit{good} in the sense of Freedman and Quinn, this means that the homeomorphism we defined on the boundary \(f\colon \partial M_0 \to \partial M_1\) extends to a homeomorphism of \(M_0\) and \(M_1\). Furthermore, since we ensured that \(f\) had a \(\nu\)-extendability condition, this homeomorphism extends to a homeomorphism \((D^4,F_0) \cong (D^4, F_1)\). Hence, by the argument previously outlined, we have that \(F_0\) and \(F_1\) are ambiently isotopic rel.\ boundary.

Thus, if we can perform surgery to improve \(W\) to an \(s\)-cobordism, we will have proven our main result. However, there is an obstruction, \(\theta(W, \bar{\nu}_W)\), in Kreck's \(\ell_5\) monoid to applying the surgery step of this argument. This must be \textit{elementary} (see \cite{Kreck}) before we can proceed.

\begin{thm}[Kreck \cite{Kreck}]\label{kreck}
If \(\theta(W, \bar{\nu}_W)\) is elementary then \((W, \bar{\nu}_W)\) is \((B, \xi)\)-bordant rel.\ boundary to an \(s\)-cobordism.   
\end{thm}

Hence, if we can show that this obstruction is elementary in the extremal normal Euler number cases of \(h=4\) and \(h=5\), we will have proven Theorem~\ref{main}.

This obstruction is usually quite difficult to calculate, but the work of Crowley and Sixt \cite{CrowSixt} gives us a way around this complexity via linking forms. This means that proving the obstruction is elementary comes down to explicit calculations, as detailed in the Appendix of \cite{og}. To understand this result, we first need some definitions.

\subsection{Quadratic forms}

\begin{defn}
     Given an abelian group \(P\), thought of as a \(\Z\)-module, we have a map given by \[\top\colon \text{Hom}_R(P,P^*) \to \text{Hom}_R(P,P^*) \colon \phi \mapsto (x\mapsto(y\mapsto \phi(y)(x))).\]
    We then define the \textit{quadratic \(Q\)-group} as \(Q_+(P) := \coker(1-\top)\) and the \textit{symmetric \(Q\)-group} as ~\(Q^+(P) := \ker(1-\top)\).
\end{defn}

\begin{defn}
   A \textit{quadratic form} is defined as a pair \((P,\psi)\) such that \(P\) is a finitely generated free abelian group and \(\psi\in Q_+(P)\).

  \begin{enumerate}
        \item The \textit{symmetrisation} of a quadratic form \((P,\psi)\) is given by \((P,(1+\top)\psi)\).
        \item An \textit{isometry} of quadratic forms is an isomorphism \(h\colon(P_1,\psi_1)\to (P_2, \psi_2)\) such that \(h^*\psi_2h = \psi_1 \in Q_+(P)\).
  \end{enumerate} 
\end{defn}

\begin{defn}
    A \textit{symmetric linking form} over \(\Z\) is defined as a pair \((T,b)\) such that \(T\) is a finite abelian group and \(b\colon T\times T \to \Q/\Z\) is a symmetric bilinear form.
\end{defn}

\begin{defn}
    A \textit{split quadratic linking form} over \(\Z\) is defined as a triple \((T,b,\nu)\) where \((T,b)\) is a symmetric linking form and \(\nu\colon T\to \Q/\Z\) is a map such that
    \begin{enumerate}
        \item \(\nu(rx) = r^2\nu(x)\),
        \item \(\nu(x_1+x_2)- \nu(x_1) - \nu(x_2) = b(x_1,x_2)\)
    \end{enumerate}
    for all \(x,x_1,x_2\in T\) and all \(r\in \Z\).
\end{defn}

\begin{defn}
    Given a nondegenerate symmetric bilinear form \(\lambda\colon V\times V\to \Z\), with adjoint \(\widehat{\lambda}\), we can obtain a linking form via \[\partial\lambda\colon \coker(\widehat{\lambda})\times\coker(\widehat{\lambda}) \to \Q/\Z \colon ([x],[y]) \mapsto \frac{1}{s}y(z) \] such that \(sx = \widehat{\lambda}(z)\) for some \(s\in \Z \backslash \{0\}\) and some \(z\in V\). This is well-defined, independent of the choice of \(s\) and \(z\), and called the \textit{boundary linking form} of the symmetric form \(\lambda\).
\end{defn}

\begin{defn}
    Let \((V,\theta)\) be a nondegenerate quadratic form over \(\Z\) with symmetrisation given by \((V,\lambda)\). The \textit{boundary split quadratic linking form} of \((V,\theta)\) is the split quadratic linking form given by \[\partial(V,\theta) := ( \coker(\widehat{\lambda}), \partial\lambda,\partial\theta)\] where \(( \coker(\widehat{\lambda}), \partial\lambda)\) is the boundary linking form of the symmetric form \((V,\lambda)\) and \[\partial\theta\colon \coker(\widehat{\lambda}) \to \Q/\Z\colon ([x])\mapsto \frac{1}{s^2}\theta(z,z)\] for \(s\) and \(z\) as above.
\end{defn}

\begin{rem}
    Eventually, we will want to calculate some of these objects. Hence we note that:

    \begin{itemize}
        \item If a nondegenerate symmetric form \((V,\lambda)\) is represented by a size \(n\) matrix \(A\), then \[\partial\lambda([x],[y])=x^TA^{-1}y,\] where \(A^{-1}\) is the inverse of \(A\) over \(\Q\) and \([x],[y]\in \Z^n/A\Z^n\).
        \item  If a nondegenerate quadratic form \((V,\theta)\) is represented by a size \(n\) matrix \(Q\), so that its symmetrisation is represented by \(A = Q+ Q^T\), then \[\partial\theta([x])= x^T(A^{-1})^TQA^{-1}x = x^TA^{-1}QA^{-1}x\] for all \([x]\in \Z^n/A\Z^n\).
    \end{itemize}
\end{rem}

\begin{defn}
    The \textit{boundary} of an isometry \(h\colon (V_0,\theta_0)\to (V_1,\theta_1)\) is the isometry \[\partial h:= (h^*)^{-1}\colon \partial(V_0,\theta_0)\to \partial(V_1,\theta_1).\]
\end{defn}

This boundary isometry determines a homomorphism \(\partial\colon \Aut(V,\theta) \to \Aut(\partial(V,\theta))\) which will be the key to our calculations. Note that there is a left action of \(\Aut(V,\theta)\) on \(\Aut(\partial(V,\theta))\) given both by \(h \cdot f:= \partial h \circ f\) and \(g\cdot f := f \circ \partial g^{-1}\) such that \(f\in \Aut(\partial(V,\theta))\) and \(g,h\in \Aut(V,\theta)\). These actions combine to give a left action of  \(\Aut(V,\theta) \times \Aut(V,\theta) \) on \(\Aut(\partial(V,\theta))\).

\begin{defn}
    The \textit{boundary automorphism set} of a nondegenerate quadratic form \((V,\theta)\) is defined as the orbit set \[\bAut(V,\theta) := \Aut(\partial(V,\theta)) / \Aut(V,\theta) \times \Aut(V,\theta).\]
\end{defn}

This is the key object to our study of the surgery obstruction. The following result will let us explicitly calculate the cases where the obstruction is elementary.

\begin{thm}[{Conway-Orson-Powell \cite{og}}] \label{trivial}
    Let \(1 \leq h \leq 8\) be the nonorientable genus of the \(\Z_2\)-surfaces \(F_0\) and \(F_1\). Then the obstruction to modified surgery on any cobordism, \((W, \Bar{\nu}_W)\), over ~\((B,\xi)\), between the exteriors \(M_i\), is elementary if the set \(\bAut(V,2\theta)\) is trivial for the nondegenerate quadratic form \[(V, 2 \theta) = \left(\mathbb{Z}^h, \left[ \pm \begin{pmatrix} 4 & 4 & 4 &\dots & 4\\  & 2 & 2 & \dots & 2 \\  &  &2& \dots &2 \\ &  & &\ddots & \vdots \\ &  & &  &2 \end{pmatrix} \right] \right) \in Q_+(\mathbb{Z}^h). \]
\end{thm}

This theorem is a summary of a lot of work in \cite{og} and is a sufficient formulation for our purposes. However it is worth speaking briefly about how this result is obtained.

By adapting work of \cite{Kreck2} in the framework of \cite{CrowSixt}, it is shown in \cite{og} (using their Proposition \(5.12\) and Lemma \(9.12\))  that if \(1 \leq h \leq 8\), the modified surgery obstruction lies in a subset of the monoid \(\ell_5(v,v) \subseteq \ell_5(\Z[\Z_2])\) defined on a specific quadratic form \(v\). They then show that it suffices to prove that the obstruction is elementary with respect to \(v\) restricted to the \(-1\) ~eigenspace. They show further that this obstruction lies in \(\ell_5(V,2\theta)\), where \((V,2\theta)\) is the nondegenerate quadratic form over \(\Z\) in the above theorem.

It is a result of Crowley and Sixt \cite[Section \(6.3\)]{CrowSixt} that if we have a nondegenerate quadratic form \((V,\theta)\) over \(\Z\) then there is a bijection between \(\ell_5(V,\theta)\) and \(\bAut(V,\theta)\). Then in \cite[Proposition \(9.15\)]{og}, it was proven that if \((V,\theta)\) is a nondegenerate quadratic form over \(\Z\) such that \(\ell_5(V,2\theta)\) is trivial, then every obstruction to modified surgery in the relevant subset of \(\ell_5\) is elementary.

Hence to show that the obstruction to modified surgery is elementary for any cobordism (as above), it suffices to show that the set \(\bAut(V,2\theta)\) is trivial.

\subsection{Calculating triviality} 
\ \\
We will show that \(\bAut(V,2\theta)\) is trivial using the following method.

\begin{itemize}
    \item First observe from the definition of \(\bAut(V,\theta)\) that an isometry \(\phi\in \Aut(\partial(V,\theta))\) is trivial in \(\bAut(V,\theta)\) if and only if \(\phi\) extends to an isometry of \((V,\theta)\).
    \item Hence we need to show that for every \(\phi\in \Aut(\partial(V,\theta))\), there is a corresponding \(\phi' \in \partial\Aut((V,\theta))\). We can do this by showing that the homomorphism 
    \[\partial\colon \Aut(V,\theta) \to \Aut(\partial(V,\theta))\]
    is surjective.
    \item As explained in \cite[Corollary \(7.42\)]{og} for a quadratic form \((V,\theta)\in Q_+(V)\) over \(\Z\) with symmetrisation \((V,\lambda)\), we have that \(\Aut(V,\theta) = \Aut(V,\lambda)\) as subsets of the automorphisms \(\Aut(V)\) of the group \(V\).
    \item Hence we will calculate \(\Aut(V,\lambda)\) with respect to some particular basis and show that the boundary map is surjective. Then we can conclude that \(\bAut(V,2\theta)\) is trivial.
\end{itemize}

We calculate this for each case \(h \in \{2,3,4,5\}\) in the next section, which allow us to prove our main result. 
\begin{proof}[Proof of Theorem~\ref{main} assuming calculations from Section \ref{sec:calc}]
    The calculations in the next section show that \(\bAut(V,2\theta)\) is trivial for the relevant quadratic forms associated to \(h=4\) and \(h=5\). Hence by Theorem~\ref{trivial}, the obstruction to modified surgery is elementary in these case. Therefore by Theorem~\ref{kreck}, and the argument above that theorem, we can conclude that \(F_0\) and \(F_1\) are ambiently isotopic rel.\ boundary.

    Now to prove the second part of the theorem, let \(F_0\) be a given \(\Z_2\)-surface in \(S^4\) with extremal Euler number and nonorientable genus \(4\) or \(5\). Puncturing this surface gives a \(\Z_2\)-surface \(F_0'\) in ~\(D^4\) with boundary component given by the puncture, thought of as the unknot. Let \(F_1\) denote the boundary connect sum of standardly embedded once-punctured copies of \(\R P^2\) with the same nonorientable genus and normal Euler number as \(F_0\), and hence as \(F_0'\). This is also a \(\Z_2\)-surface in \(D^4\) with boundary the unknot. Note that the unknot \(U\) is such that \(\vert \det(U) \vert =1\), so we can apply the first part of the theorem to \(F_0'\) and \(F_1\).

    In particular we get that \(F_0'\) and \(F_1\) are ambiently isotopic rel.\ boundary. We now fill back in the puncture and extend this isotopy using a constant isotopy in the removed \(D^4\), which gives an isotopy of \(F_0\) to the connect sum of copies of \(\R P^2\) standardly embedded in \(S^4\). This is precisely the definition of \(F_0\) being topologically unknotted.
\end{proof}

\section{Explicit Calculations}\label{sec:calc}

In this section, we perform the calculations needed in proving our main results using the method outlined above, for \(h\in \{2,3,4,5\}\). 
\begin{itemize}
    \item We generate a list of automorphisms of \((V,2\theta)\) and take \(\partial\) of these to obtain the image \(\partial(\Aut(V, 2\theta))\).
    \item We generate a list of automorphisms of \(\partial(V, 2\theta)\) by taking all matrices which fulfil the equations coming from the boundary quadratic form and reducing this list of matrices down to \(\Aut(\partial(V, 2\theta))\) by checking which ones also preserve the symmetric linking form.
    \item Finally, we check that this list is the same as the image, and conclude that the boundary map is surjective.
\end{itemize}

In \cite{og}, Sage is used to generate the isometries in the image, but the isometries in \(\Aut(\partial(V,\theta))\) are calculated by hand. As nonorientable genus increases, the number of potential automorphisms increases drastically, so this task becomes impossible. We use Sage to automate this step, and so can conclude results for higher values of \(h\).

We first reformulate the \(h=2\) and \(h=3\) cases already calculated in order to demonstrate our method, before proving the new \(h=4\) and \(h=5\) cases.

\subsection{Calculation for \(h=2\)}

\begin{prop}
Given the quadratic
form \[(V, 2 \theta) = \left(\mathbb{Z}^2, \left[ \pm \begin{pmatrix} 4 & 4 \\ 0 & 2\end{pmatrix} \right] \right) \in Q_+(\mathbb{Z}^2), \]
the set \(\bAut(V, 2\theta)\) is trivial.
\end{prop}

\begin{proof}
Note that for any quadratic form \((V,\theta)\), we have that \(\bAut(V,\theta) = \bAut(V,-\theta)\), so from now on (in every calculation) we will just work with the positive case. We begin by applying the formulas and definitions from the previous section to our particular situation. 

\begin{lem} In the case of \(h=2\), we have that:
\begin{enumerate}
    \item The quadratic form \(\psi := \left[ \begin{pmatrix} 4 & 4 \\ 0 & 2\end{pmatrix} \right] \) is represented by \( \widetilde{Q} = \begin{pmatrix} 2 & 0 \\ 0 & 2 \end{pmatrix}.\)
    \item The cokernel of the adjoint \(\widehat{\lambda}\) of the symmetrisation \(\lambda\) of \(\psi\) is isomorphic to \((\mathbb{Z}_4)^2.\)
    \item The boundary symmetric form is \(\partial\lambda\colon (\mathbb{Z}_4)^2 \times (\mathbb{Z}_4)^2 \to \mathbb{Q}/\mathbb{Z}\colon ([x],[y]) \mapsto x^T \begin{pmatrix} \frac{1}{4} & 0 \\ 0 & \frac{1}{4} \end{pmatrix} y.\)
    \item The boundary quadratic form is \(\partial\psi\colon (\mathbb{Z}_4)^2 \to \mathbb{Q}/\mathbb{Z}\colon \left(\left[\begin{smallmatrix} x_1 \\ x_2\end{smallmatrix}\right]\right) \mapsto \frac{1}{8}(x_1^2+x_2^2).\)
\end{enumerate}
\end{lem}

\begin{proof}
    Write \(Q = \begin{pmatrix} 4 & 4 \\ 0 & 2\end{pmatrix}\), and define \(A:= Q + Q^T\). Performing row and column operations on \(A\) one gets \(\coker(A)  = (\mathbb{Z}_4)^2\) with generators given by the classes of \((0,1), (1,1)\) with respect to the canonical basis for \(\coker(A)\).
    Under this generating set, the boundary symmetric linking form ~\(\partial\lambda([x],[y]) = x^T A^{-1} y\) is isometric to the pairing
    \[\partial\lambda([x],[y]) = x^T \begin{pmatrix} \frac{1}{4} & 0 \\ 0 & \frac{1}{4} \end{pmatrix} y.\]
    Under this generating set \(Q\) becomes: \[ \begin{pmatrix} 2 & 2 \\ -2 & 2\end{pmatrix} \thicksim \begin{pmatrix} 2 & 0 \\ 0 & 2\end{pmatrix} = \widetilde{Q}\] by the defining relation of \(Q_+(\mathbb{Z}^2)\). Finally, for \(x = \left[\begin{smallmatrix} x_1 \\ x_2\end{smallmatrix}\right]\) the boundary quadratic form is given by \[ \partial\psi(x) = x^T A^{-1}QA^{-1} x = \frac{1}{8}(x_1^2+x_2^2). \qedhere \]

\end{proof}

Using Sage (see Appendix), we list the automorphisms of \(\psi\), represented by \(\widetilde{Q}\), as matrices with coefficients in \((\Z_4)^2\). We find that the automorphism group \(\Aut(V,\psi)\) has \(8\) elements. From this, we use Sage to generate the elements of \(\im(\partial)\) for
\[\partial\colon \Aut(V,\psi) \to \Aut(\partial(V,\psi))\]
in the same basis. In particular, by the definition of \(\partial\), we can do this by taking inverse transpose of each element in \(\Aut(V,\psi)\). 

We now focus our attention on generating the set \(\Aut(\partial(V,\psi))\) and comparing it to the image. An automorphism \(f\) of \(\partial(V,\psi)\) is determined by its images on the canonical generators of \((\mathbb{Z}_4)^2\).
    \begin{itemize}
        \item For \((a,b) := f(1,0)\), the formula above gives: \[\frac{1}{8} = \frac{a^2+b^2}{8} \in \mathbb{Q}/\mathbb{Z}.\] So we get the first equation \[a^2+b^2 \equiv 1 \mod 8.\]
        \item For \((a,b) := f(0,1)\), the formula above gives: \[\frac{1}{8} = \frac{a^2+b^2}{8} \in \mathbb{Q}/\mathbb{Z}.\] So we get the second equation, which is also \[a^2+b^2 \equiv 1 \mod 8.\]
    \end{itemize}

Hence we use Sage to generate a list of matrices whose \(i^\text{th}\) column is a solution set of the \(i^\text{th}\) equation above (in this case \(i\in \{1,2\}\), in general we will have \(i\in \{1,\dots, h\}\)).  We reduce the coefficients of these matrices to have entries in \((\Z_4)^2\) and eliminate any resulting duplicates. Then we filter out the matrices which do not preserve the boundary symmetric linking form, and we are left with the set \(\Aut(\partial(V,2\theta))\).
    
Again using Sage, we compare the two lists we have generated, and find that \(\im(\partial) = \Aut(\partial(V,2\theta))\). This concludes the proof.
\end{proof}

\subsection{Calculation for \(h=3\)}

\begin{prop}
Given the quadratic
form \[(V, 2 \theta) = \left(\mathbb{Z}^3, \left[\pm \begin{pmatrix} 4 & 4 &4\\ 0 & 2&2 \\ 0&0&2\end{pmatrix} \right] \right) \in Q_+(\mathbb{Z}^3). \]
the set \(\bAut(V, 2\theta)\) is trivial.
\end{prop}

\begin{proof}
We use the same method as the previous case, and begin with a lemma.

\begin{lem} In the case of \(h=2\), we have that:
\begin{enumerate}
    \item The quadratic form \(\psi := \left[ \begin{pmatrix} 4 & 4 &4\\ 0 & 2&2 \\ 0&0&2\end{pmatrix} \right] \) is represented by \( \widetilde{Q} = \begin{pmatrix} 12 & 8 &-8\\ 0 & 2&-2 \\ 0&0&2 \end{pmatrix}.\)

    \item The cokernel of the adjoint \(\widehat{\lambda}\) of the symmetrisation \(\lambda\) of \(\psi\) is isomorphic to \(\mathbb{Z}_8 \oplus (\mathbb{Z}_2)^2.\)

    \item The boundary symmetric form is \[\partial\lambda\colon \mathbb{Z}_8 \oplus (\mathbb{Z}_2)^2\times \mathbb{Z}_8 \oplus (\mathbb{Z}_2)^2\to \mathbb{Q}/\mathbb{Z}\colon ([x],[y]) \mapsto x^T \begin{pmatrix} \frac{3}{8} & \frac{1}{2} &\frac{1}{2}\\ \frac{1}{2} & 0&\frac{1}{2} \\ \frac{1}{2}&\frac{1}{2}&0 \end{pmatrix}y.\] 

    \item The boundary quadratic form is \[\partial\psi\colon \mathbb{Z}_8 \oplus (\mathbb{Z}_2)^2 \to \mathbb{Q}/\mathbb{Z}\colon \left(\left[\begin{smallmatrix} x_1 \\ x_2 \\ x_3\end{smallmatrix}\right]\right) \mapsto \frac{3}{16}x_1^2+\frac{1}{2}(x_2^2+x_3^2+ x_1x_2 + x_1x_3+x_2x_3).\]
\end{enumerate}
\end{lem}

\begin{proof}
    Write \(Q = \begin{pmatrix} 4 & 4 &4\\ 0 & 2&2 \\ 0&0&2\end{pmatrix}\), and define \(A:= Q + Q^T\). Performing row and column operations on \(A\) one gets \(\coker(A)  = \mathbb{Z}_8 \oplus (\mathbb{Z}_2)^2\) with generators given by the classes of \((1,0,0), (0,1,1),(2,0,1)\) with respect to the canonical basis for \(\coker(A)\).
    Under this generating set, the boundary symmetric linking form \(\partial\lambda([x],[y]) = x^T A^{-1} y\) is isometric to the pairing
    \[\partial\lambda([x],[y]) = x^T \begin{pmatrix} \frac{3}{8} & \frac{1}{2} &\frac{1}{2}\\ \frac{1}{2} & 0&\frac{1}{2} \\ \frac{1}{2}&\frac{1}{2}&0 \end{pmatrix} y.\]
    Under this generating set \(Q\) becomes: \[ \begin{pmatrix} 12 & 8 &-4\\ 0 & 2&0 \\ -4&-2&2\end{pmatrix} \thicksim \begin{pmatrix} 12 & 8 &-8\\ 0 & 2&-2 \\ 0&0&2 \end{pmatrix} = \widetilde{Q}\] by the defining relation of \(Q_+(\mathbb{Z}^3)\). Finally, for \(x = \left[\begin{smallmatrix} x_1 \\ x_2 \\ x_3\end{smallmatrix}\right]\) the boundary quadratic form is given by \[ \partial\psi(x) = x^T A^{-1}QA^{-1} x = \frac{3}{16}x_1^2+\frac{1}{2}(x_2^2+x_3^2+ x_1x_2 + x_1x_3+x_2x_3). \qedhere \]

\end{proof}

In this case, the coefficients of the matrices representing these automorphisms are more complicated. As explained in \cite[Theorem \(3.6\)]{Coefish} and \cite{og}, we have a canonical isomorphism
\[\Aut(\Z_8\oplus\Z_2\oplus\Z_2) \cong \frac{\left\{A = \begin{pmatrix} a & 4b &4c\\ d &e&f \\ g&h&i \end{pmatrix} \in \text{M}_3(\Z) \Bigg{\vert} A \text{ mod } 2\in \text{GL}_3(\Z_2)\right\}}{\left\{\begin{pmatrix} 8a & 8b &8c\\ 2d &2e&2f \\ 2g&2h&2i \end{pmatrix} \Bigg{\vert} a,b,c,d,e,f,g,h,i\in\Z\right\}}\]
which we will use implicitly throughout this proof. We will also use this result (for different cokernel) implicitly in the next two cases as well. 

Using Sage, we list the automorphisms of \(\psi\), represented by \(\widetilde{Q}\), reducing the elements according the result above. We find that the automorphism group \(\Aut(V,\psi)\) has \(48\) elements. From this, we generate the elements of \(\im(\partial)\) in the same basis, as in our previous case.

We now focus our attention on generating the set \(\Aut(\partial(V,\psi))\) and comparing it to the image. An automorphism \(f\) of \(\partial\psi\) is determined by its images on the canonical generators of \(\mathbb{Z}_8 \oplus (\mathbb{Z}_2)^2\). Hence we proceed in the same way as the previous case.
    \begin{itemize}
        \item For \((a,b,c) = f(1,0,0)\), the formula above for the boundary quadratic form gives:
        \[ \frac{3}{16} = \frac{3}{16}x_1^2+\frac{8}{16}(x_2^2+x_3^2+ x_1x_2 + x_1x_3+x_2x_3) \in \mathbb{Q}/\mathbb{Z}.\]
        So we get the first equation \[3x_1^2+8(x_2^2+x_3^2+ x_1x_2 + x_1x_3+x_2x_3) \equiv 3 \mod 16.\]
        \item For \((a,b,c) = f(0,1,0)\), We get the second equation \[3x_1^2+8(x_2^2+x_3^2+ x_1x_2 + x_1x_3+x_2x_3)\equiv 8 \mod 16.\]
        \item Similarly for \((a,b,c) = f(0,0,1)\), We get the third equation \[3x_1^2+8(x_2^2+x_3^2+ x_1x_2 + x_1x_3+x_2x_3) \equiv 8 \mod 16.\]
    \end{itemize}
Hence we use Sage to generate a list of matrices whose \(i^\text{th}\) column is a solution set of the \(i^\text{th}\) equation above. We reduce the coefficients of these matrices according to the result above and eliminate any duplicates. Then we filter out the matrices which do not preserve the boundary symmetric linking form, and we are left with the set \(\Aut(\partial(V,2\theta))\). Again using Sage, we compare the two lists we have generated, and find that \(\im(\partial) = \Aut(\partial(V,2\theta))\). This concludes the proof.
\end{proof}

\subsection{Calculation for \(h=4\)}

\begin{prop}
Given the quadratic
form \[(V, 2 \theta) = \left(\mathbb{Z}^4, \left[\pm \begin{pmatrix} 4 & 4 &4&4\\ 0 & 2&2&2 \\ 0&0&2&2 \\ 0&0&0&2\end{pmatrix} \right] \right) \in Q_+(\mathbb{Z}^4). \]
the set \(\bAut(V, 2\theta)\) is trivial.
\end{prop}

\begin{proof}
Once more, we begin with a lemma.

\begin{lem} In the case of \(h=4\), we have that:
\begin{enumerate}
    \item The quadratic form \(\psi := \left[ \begin{pmatrix}4 & 4 &4&4\\ 0 & 2&2&2 \\ 0&0&2&2 \\ 0&0&0&2\end{pmatrix} \right] \) is represented by \(\widetilde{Q} = \begin{pmatrix} 12&-28&8&-8\\ 0&20&-12&8 \\ 0&0&2&-2 \\ 0&0&0&2 \end{pmatrix}.\) 

    \item The cokernel of the adjoint \(\widehat{\lambda}\) of the symmetrisation \(\lambda\) of \(\psi\) is isomorphic to \((\mathbb{Z}_4)^2 \oplus (\mathbb{Z}_2)^2.\)

    \item The boundary symmetric form is \[\partial\lambda\colon (\mathbb{Z}_4)^2 \oplus (\mathbb{Z}_2)^2\times(\mathbb{Z}_4)^2 \oplus (\mathbb{Z}_2)^2\to \mathbb{Q}/\mathbb{Z}\colon ([x],[y]) \mapsto x^T \begin{pmatrix} \frac{1}{2}&\frac{1}{4}&0&\frac{1}{2}\\ \frac{1}{4}&\frac{1}{2}&0&0 \\ 0&0&0&\frac{1}{2} \\ \frac{1}{2}&0&\frac{1}{2}&0  \end{pmatrix}y.\] 

    \item The boundary quadratic form is \[\partial\psi\colon (\mathbb{Z}_4)^2 \oplus (\mathbb{Z}_2)^2 \to \mathbb{Q}/\mathbb{Z}\colon \left(\left[\begin{smallmatrix} x_1 \\ x_2 \\ x_3 \\ x_4\end{smallmatrix}\right]\right) \mapsto \frac{1}{4}x_1^2+\frac{1}{4}x_2^2+\frac{1}{4}x_1x_2+\frac{1}{2}(x_3^2+x_4^2+x_4x_1+x_4x_3).\]
\end{enumerate}
\end{lem}

\begin{proof}
    Write \(Q = \begin{pmatrix} 4 & 4 &4&4\\ 0 & 2&2&2 \\ 0&0&2&2 \\ 0&0&0&2 \end{pmatrix}\), and define \(A:= Q + Q^T\). Performing row and column operations on \(A\) one gets \(\coker(A)  = (\mathbb{Z}_4)^2 \oplus (\mathbb{Z}_2)^2\) with generators given by the classes of \((2,1,1,1), (1,1,1,1),(0,2,1,1),(2,1,1,0)\) with respect to the canonical basis for \(\coker(A)\).
    Under this generating set, the boundary symmetric linking form \(\partial\lambda([x],[y]) = x^T A^{-1} y\) is isometric to the pairing
    \[\partial\lambda([x],[y]) = x^T \begin{pmatrix} \frac{1}{2}&\frac{1}{4}&0&\frac{1}{2}\\ \frac{1}{4}&\frac{1}{2}&0&0 \\ 0&0&0&\frac{1}{2} \\ \frac{1}{2}&0&\frac{1}{2}&0  \end{pmatrix} y.\]
    Under this generating set \(Q\) becomes: \[\begin{pmatrix} 12&-16&6&-2\\ -12&20&-8&0 \\ 2&-4&2&0 \\ -6&8&-2&2\end{pmatrix} \thicksim \begin{pmatrix} 12&-28&8&-8\\ 0&20&-12&8 \\ 0&0&2&-2 \\ 0&0&0&2 \end{pmatrix} = \widetilde{Q}\] by the defining relation of \(Q_+(\mathbb{Z}^4)\). Finally, for \(x = \left[\begin{smallmatrix} x_1 \\ x_2 \\ x_3 \\ x_4\end{smallmatrix}\right]\) the boundary quadratic form is given by \[ \partial\psi(x) = x^T A^{-1}QA^{-1} x = \frac{1}{4}x_1^2+\frac{1}{4}x_2^2+\frac{1}{4}x_1x_2+\frac{1}{2}(x_3^2+x_4^2+x_4x_1+x_4x_3). \qedhere \]

\end{proof}
Using Sage, we list the automorphisms of \(\psi\), represented by \(\widetilde{Q}\), reducing the elements according to \cite[Theorem \(3.6\)]{Coefish}. We find that the automorphism group \(\Aut(V,\psi)\) has \(1024\) elements. From this, we generate the elements of \(\im(\partial)\) in the same basis, as in our previous case.

We now focus our attention on generating the set \(\Aut(\partial(V,\psi))\) and comparing it to the image. An automorphism \(f\) of \(\partial\psi\) is determined by its images on the canonical generators of \((\mathbb{Z}_4)^2 \oplus (\mathbb{Z}_2)^2\). Hence we proceed in the same way as the previous cases.
    \begin{itemize}
        \item For \((a,b,c,d) = f(1,0,0,0)\), the formula above for the boundary quadratic form gives:
        \[ \frac{1}{4} = \frac{1}{4}x_1^2+\frac{1}{4}x_2^2+\frac{1}{4}x_1x_2+\frac{2}{4}(x_3^2+x_4^2+x_4x_1+x_4x_3) \in \mathbb{Q}/\mathbb{Z}.\]
        So we get the first equation \[x_1^2+x_2^2+x_1x_2+2(x_3^2+x_4^2+x_4x_1+x_4x_3) \equiv 1 \mod 4.\]
        \item Similarly for \((a,b,c,d) = f(0,1,0,0)\), we get the second equation \[x_1^2+x_2^2+x_1x_2+2(x_3^2+x_4^2+x_4x_1+x_4x_3) \equiv 1 \mod 4.\]
        \item For \((a,b,c,d) = f(0,0,1,0)\), we get the third equation \[x_1^2+x_2^2+x_1x_2+2(x_3^2+x_4^2+x_4x_1+x_4x_3) \equiv 2 \mod 4.\]
        \item For \((a,b,c,d) = f(0,0,0,1)\), we get the fourth equation \[x_1^2+x_2^2+x_1x_2+2(x_3^2+x_4^2+x_4x_1+x_4x_3) \equiv 2 \mod 4.\]
        
    \end{itemize}
Hence we use Sage to generate a list of matrices whose \(i^\text{th}\) column is a solution set of the \(i^\text{th}\) equation above. We reduce the coefficients of these matrices according to \cite{Coefish} and eliminate any duplicates. Then we filter out the matrices which do not preserve the boundary symmetric linking form, and we are left with the set \(\Aut(\partial(V,2\theta))\). Again using Sage, we compare the two lists we have generated, and find that \(\im(\partial) = \Aut(\partial(V,2\theta))\). This concludes the proof. \qedhere
\end{proof}

\subsection{Calculation for \(h=5\)}

\begin{prop}
Given the quadratic
form \[(V, 2 \theta) = \left(\mathbb{Z}^5, \left[\pm \begin{pmatrix} 4 & 4 &4&4&4\\ 0 & 2&2&2&2 \\ 0&0&2&2&2 \\ 0&0&0&2&2 \\ 0&0&0&0&2\end{pmatrix} \right] \right) \in Q_+(\mathbb{Z}^5). \]
the set \(\bAut(V, 2\theta)\) is trivial.
\end{prop}

\begin{proof}
We begin with a lemma.

\begin{lem} In the case of \(h=5\), we have that:
\begin{enumerate}
    \item The quadratic form \(\psi := \left[ \begin{pmatrix} 4 & 4 &4&4&4\\ 0 & 2&2&2&2 \\ 0&0&2&2&2 \\ 0&0&0&2&2 \\ 0&0&0&0&2\end{pmatrix} \right] \) is represented by \[\widetilde{Q} = \begin{pmatrix} 180 & 40 &-32&-32&-16\\ 0 & 4&-2&-4&-4 \\ 0&0&2&2&0 \\ 0&0&0&2&2 \\ 0&0&0&0&2\end{pmatrix}.\]

    \item The cokernel of the adjoint \(\widehat{\lambda}\) of the symmetrisation \(\lambda\) of \(\psi\) is isomorphic to \(\mathbb{Z}_8 \oplus (\mathbb{Z}_2)^4.\)

    \item The boundary symmetric form is \[\partial\lambda\colon \mathbb{Z}_8 \oplus (\mathbb{Z}_2)^4\times\mathbb{Z}_8 \oplus (\mathbb{Z}_2)^4\to \mathbb{Q}/\mathbb{Z}\colon ([x],[y]) \mapsto x^T \begin{pmatrix} \frac{5}{8}&0&\frac{1}{2}&0&\frac{1}{2} \\ 0&0&\frac{1}{2}&0&\frac{1}{2} \\ \frac{1}{2}&\frac{1}{2}&0&0&0 \\ 0&0&0&0&\frac{1}{2} \\ \frac{1}{2}&\frac{1}{2}&0&\frac{1}{2}&0 \end{pmatrix}y.\] 

    \item The boundary quadratic form is \[\partial\psi\colon(\mathbb{Z}_8 \oplus (\mathbb{Z}_2)^4 \to \mathbb{Q}/\mathbb{Z}\colon \left(\left[\begin{smallmatrix} x_1 \\ x_2 \\ x_3 \\ x_4\\ x_5\end{smallmatrix}\right]\right) \mapsto \frac{5}{16}x_1^2+\frac{1}{2}(x_4^2+x_5^2+x_1x_3+x_1x_5+x_2x_3+x_2x_5+x_4x_5).\]
\end{enumerate}

\end{lem}

\begin{proof}
    Write \(Q =\begin{pmatrix} 4 & 4 &4&4&4\\ 0 & 2&2&2&2 \\ 0&0&2&2&2 \\ 0&0&0&2&2 \\ 0&0&0&0&2\end{pmatrix}\), and define \(A:= Q + Q^T\). Performing row and column operations on \(A\) one gets \(\coker(A)  = \mathbb{Z}_8 \oplus (\mathbb{Z}_2)^4\) with generators given by the classes of \((1,0,0,0,0), (0,1,1,1,1),(6,1,0,0,0),(4,0,0,1,1),(2,1,1,1,0)\) with respect to the canonical basis for \(\coker(A)\).
    Under this generating set, the boundary symmetric linking form \(\partial\lambda([x],[y]) = ~x^T A^{-1} y\) is isometric to the pairing
    \[\partial\lambda([x],[y]) = x^T \begin{pmatrix} \frac{5}{8}&0&\frac{1}{2}&0&\frac{1}{2} \\ 0&0&\frac{1}{2}&0&\frac{1}{2} \\ \frac{1}{2}&\frac{1}{2}&0&0&0 \\ 0&0&0&0&\frac{1}{2} \\ \frac{1}{2}&\frac{1}{2}&0&\frac{1}{2}&0 \end{pmatrix} y.\]
    Under this generating set \(Q\) becomes: \[\widetilde{Q} = \begin{pmatrix} 180 & 40 &-32&-32&-16\\ 0 & 4&-2&-4&-4 \\ 0&0&2&2&0 \\ 0&0&0&2&2 \\ 0&0&0&0&2\end{pmatrix} \] using the defining relation of \(Q_+(\mathbb{Z}^5)\). Finally, for \(x = \left[\begin{smallmatrix} x_1 \\ x_2 \\ x_3 \\ x_4 \\x_5\end{smallmatrix}\right]\), the boundary quadratic form is \[ \partial\psi(x) = x^T A^{-1}QA^{-1} x = \frac{5}{16}x_1^2+\frac{1}{2}(x_4^2+x_5^2+x_1x_3+x_1x_5+x_2x_3+x_2x_5+x_4x_5). \qedhere \]
\end{proof}

Using Sage, we list the automorphisms of \(\psi\), represented by \(\widetilde{Q}\), reducing the elements according to \cite[Theorem \(3.6\)]{Coefish}. We find that the automorphism group \(\Aut(V,\psi)\) has \(3840\) elements. From this, we generate the elements of \(\im(\partial)\) in the same basis, as in our previous case.

We now focus our attention on generating the set \(\Aut(\partial(V,\psi))\) and comparing it to the image. An automorphism \(f\) of \(\partial\psi\) is determined by its images on the canonical generators of \((\mathbb{Z}_8 \oplus (\mathbb{Z}_2)^4\). We proceed in the same way as the other cases.
    \begin{itemize}
        \item For \((a,b,c,d,e) = f(1,0,0,0,0)\), the formula above for the boundary quadratic form gives:
        \[ \frac{5}{16} = \frac{5}{16}x_1^2+\frac{8}{16}(x_4^2+x_5^2+x_1x_3+x_1x_5+x_2x_3+x_2x_5+x_4x_5) \in \mathbb{Q}/\mathbb{Z}.\]
        So we get the first equation \[5x_1^2+8(x_4^2+x_5^2+x_1x_3+x_1x_5+x_2x_3+x_2x_5+x_4x_5) \equiv 5 \mod 16.\]
        \item Similarly, for \((a,b,c,d,e) = f(0,1,0,0,0)\), we get the second equation \[5x_1^2+8(x_4^2+x_5^2+x_1x_3+x_1x_5+x_2x_3+x_2x_5+x_4x_5) \equiv 0 \mod 16.\]
        \item For \((a,b,c,d,e) = f(0,0,1,0,0)\), we get the third equation \[5x_1^2+8(x_4^2+x_5^2+x_1x_3+x_1x_5+x_2x_3+x_2x_5+x_4x_5) \equiv 0 \mod 16.\]
         \item For \((a,b,c,d,e) = f(0,0,0,1,0)\), we get the fourth equation \[5x_1^2+8(x_4^2+x_5^2+x_1x_3+x_1x_5+x_2x_3+x_2x_5+x_4x_5) \equiv 8 \mod 16.\]
         \item For \((a,b,c,d,e) = f(0,0,0,0,1)\), we get the fifth equation \[5x_1^2+8(x_4^2+x_5^2+x_1x_3+x_1x_5+x_2x_3+x_2x_5+x_4x_5) \equiv 8 \mod 16.\]
    \end{itemize}
Hence we use Sage to generate a list of matrices whose \(i^\text{th}\) column is a solution set of the \(i^\text{th}\) equation above. We reduce the coefficients of these matrices according to \cite{Coefish} and eliminate any duplicates. Then we filter out the matrices which do not preserve the boundary symmetric linking form, and we are left with the set \(\Aut(\partial(V,2\theta))\). Again using Sage, we compare the two lists we have generated, and find that \(\im(\partial) = \Aut(\partial(V,2\theta))\). This concludes the proof.
\end{proof}

Hence our main results follow from the proof given in the previous section.

\section{Comments and Conclusions}\label{sec:comm}

As we have seen throughout the paper, this method is easily generalised to larger values of ~\(h\). The main problem in testing these higher cases is the exponential growth in the number of calculations we need to perform.

In each case of \(h\), we must generate two lists of \(h\times h\) matrices. The image and the boundary automorphisms. The number of elements of the image is the number of automorphisms of the quadratic form, which is known to be \(2^h\cdot h!\), except for \(h=4\) where there are more due to extra symmetries. The presence of a factorial in this expression is especially problematic when we want to compute larger values. Not only must we generate this list, we need to take the inverse transpose of each matrix and reduce the coefficients to get the image.

Obtaining the list of boundary automorphisms is more taxing computationally. We generate solutions to equations and look at all possible combinations of matrices built from the solutions. Adding just one new column (that is, going from \(h\) to \(h+1\)) adds exponentially more possible permutations. This list by definition needs to be at least as big as the image, and is often times significantly bigger. This is manageable in the cases we have calculated, but becomes unreasonable even by the time you reach \(h=8\), where each increase in nonorientable genus means the number of matrices we need to compute increases by a factor of billions. This proves to be too much of a strain on both time and memory.

We can still comment on borderline cases like \(h=6,7\) by slightly adapting the code. In particular, instead of taking all possible combinations of every solution to the equations as our total matrix, we sequentially generated the matrix by checking the conditions as we added combinations. This combined with utilising faster computers and running computations in parallel over many different cores, we were able to conclude results in these cases.

\begin{prop}
    The set \(\bAut(V,2\theta)\) is non-trivial in the cases of \(h=6\) and \(h=7\).
\end{prop}

In particular this means that we cannot conclude anything about knottedness of \(\Z_2\)-surfaces with extremal normal Euler number \(6\) and \(7\).
\begin{itemize}
    \item In the case of \(h=6\), the number of boundary automorphisms is double the number of elements of the image. That is, \(\im(\partial)\) is an index two subgroup of \(\Aut(\partial(V,\theta))\).
    \item In the case of \(h=7\), \(\im(\partial)\) is an index eight subgroup of \(\Aut(\partial(V,\theta))\).
\end{itemize}
Hence for \(h=6\) and \(h=7\), the boundary map is not surjective, so the method fails. We conjecture that the method fails for all \(h\geq 8\) as well.

It would be interesting to know if there is a general argument which shows that Conjecture~\ref{conj} is true in every case or if the failure of this method is indicative of the failure of the conjecture.

\appendix
\section{Description of the Sage Code}
The code used in the calculations is explained here, using \(h=3\) as an example: \\
We first obtain \(\im(\partial)\) by entering in \(\widetilde{Q}\), listing all the automorphisms of \(\widetilde{Q}\), taking the inverse transpose, and reducing the coefficients according to \cite{Coefish}.

\begin{verbatim}
Q = QuadraticForm(ZZ, 3, [12,8,-8,2,-2,2])

    Q.number_of_automorphisms()

    L = Q.automorphisms()
    N = [x.inverse().transpose() for x in L]

    for i in N:
    i[0,0] = Mod(i[0,0],8)
    i[0,1] = Mod(i[0,1],8)
    i[0,2] = Mod(i[0,2],8)
    i[1,0] = Mod(i[1,0],2)
    i[1,1] = Mod(i[1,1],2)
    i[1,2] = Mod(i[1,2],2)
    i[2,0] = Mod(i[2,0],2)
    i[2,1] = Mod(i[2,1],2)
    i[2,2] = Mod(i[2,2],2)

    N1 = []
    for x in N:
    if x[0][1] % 4 == 0 and x[0][2] %4 == 0: N1.append(x)
    else: pass

    Image = []
    for i in N1:
    if matrix([[Mod(i[0,0],2), Mod(i[0,1],2), Mod(i[0,2],2)],
    [Mod(i[1,0],2), Mod(i[1,1],2), Mod(i[1,2],2)],
    [Mod(i[2,0],2), Mod(i[2,1],2), Mod(i[2,2],2)]]).determinant() == 0: pass
    else: Image.append(i)
\end{verbatim}

We define a function which will reduce entries of matrices to \(\mathbb{Q}/\mathbb{Z}\).

\begin{verbatim}
def QZ(a): return (a.numerator() % a.denominator())/a.denominator()
\end{verbatim}

 We then plug in and solve the equations obtained from each canonical generator to generate the total list of possible automorphisms. 
\begin{verbatim}
    eq1 = []
    for a in Integers(8):
    for b in Integers(2):
    for c in Integers(2):
    a1 = matrix(ZZ, [a])[0][0]
    b1 = matrix(ZZ, [b])[0][0]
    c1 = matrix(ZZ, [c])[0][0]
    if (3*a1^2 + 8*(b1^2 + c1^2 + a1*b1 + a1*c1 +b1*c1)) % 16 == 3:
    eq1.append([a1,b1,c1])

    eq2 = []
    for a in Integers(8):
    for b in Integers(2):
    for c in Integers(2):
    a1 = matrix(ZZ, [a])[0][0]
    b1 = matrix(ZZ, [b])[0][0]
    c1 = matrix(ZZ, [c])[0][0]
    if (3*a1^2 + 8*(b1^2 + c1^2 + a1*b1 + a1*c1 +b1*c1)) % 16 == 8:
    eq2.append([a1,b1,c1])

    eq3 = []
    for a in Integers(8):
    for b in Integers(2):
    for c in Integers(2):
    a1 = matrix(ZZ, [a])[0][0]
    b1 = matrix(ZZ, [b])[0][0]
    c1 = matrix(ZZ, [c])[0][0]
    if (3*a1^2 + 8*(b1^2 + c1^2 + a1*b1 + a1*c1 +b1*c1)) % 16 == 8:
    eq3.append([a1,b1,c1])

\end{verbatim}
Notice that in this case equations two and three are the same, but to demonstrate the method, we will treat them as separate. Taking out the matrices which do not preserve the boundary symmetric linking form, we are left with the set \(\Aut(\partial(V,2\theta))\).

\begin{verbatim}
    Z = []
    for i in eq1:
    for j in eq2:
    for k in eq3:
    Z.append(matrix(QQ, [[i[0],j[0],k[0]], [i[1],j[1],k[1]], [i[2],j[2],k[2]]]))

    Z1 = []
    for i in Z:
    if matrix([[Mod(i[0,0],2), Mod(i[0,1],2), Mod(i[0,2],2)],
    [Mod(i[1,0],2), Mod(i[1,1],2), Mod(i[1,2],2)],
    [Mod(i[2,0],2), Mod(i[2,1],2), Mod(i[2,2],2)]]).determinant() == 0: pass
    else: Z1.append(i)

    Automorphisms = []
    for K in Z1:
    W = K.transpose() * LinkingForm * K
    T = matrix([[QZ(W[0,0]), QZ(W[0,1]), QZ(W[0,2])],
    [QZ(W[1,0]), QZ(W[1,1]), QZ(W[1,2])],
    [QZ(W[2,0]), QZ(W[2,1]), QZ(i[2,2])]])
    if T == LinkingForm: Automorphisms.append(K)
    else: pass
\end{verbatim}

Finally we check if the image is the same as the set of boundary automorphisms. 

\begin{verbatim}
    list(Set(Automorphisms).intersection(Set(Image))) == Automorphisms
    list(Set(Image).intersection(Set(Automorphisms))) == Image
\end{verbatim}

The second command should always be true, since we have that \(\im(\partial) \subseteq \Aut\). This provides a way to check for errors in the code. The first command is true if they are equal, false if they are not. Hence an output of ``True, True'' verifies the claim that \(\im(\partial) = \Aut(\partial(V,2\theta))\).

\def\MR#1{}
\bibliography{bibliography}

\begin{bibdiv}
\begin{biblist}

\bib{og}{article}{
      author={Conway, Anthony},
      author={Orson, Patrick},
      author={Powell, Mark},
       title={Unknotting nonorientable surfaces},
        date={2023},
        note={\url{https://arxiv.org/abs/2306.12305}},
}

\bib{CrowSixt}{article}{
      author={Crowley, Diarmuid},
      author={Sixt, Jörg},
       title={Stably diffeomorphic manifolds and \(l_{2q+1}(\textbf{Z}[\pi])\)},
        date={2011},
     journal={Forum Math.},
      volume={\textbf{23}},
      number={3},
       pages={483\ndash 538},
}

\bib{FQ}{book}{
      author={Freedman, Michael~H.},
      author={Quinn, Frank},
       title={Topology of 4-manifolds},
   publisher={Princeton University Press},
        date={1990},
}

\bib{Coefish}{article}{
      author={Hillar, Christopher~J.},
      author={Rhea, Darren~L.},
       title={Automorphisms of finite abelian groups},
        date={2007},
     journal={The American Mathematical Monthly},
      volume={\textbf{114}},
      number={10},
       pages={917\ndash 923},
}

\bib{Kreck2}{article}{
      author={Kreck, Matthias},
       title={On the homeomorphism classification of smooth knotted surfaces in the 4-sphere},
        date={1990},
     journal={Geometry of low-dimensional manifolds: 1},
       pages={707\ndash 754},
}

\bib{Kreck}{article}{
      author={Kreck, Matthias},
       title={Surgery and duality},
        date={1999},
     journal={Annals of Mathematics},
      volume={(2) \textbf{149}},
      number={3},
       pages={707\ndash 754},
}

\bib{Law}{article}{
      author={Lawson, Terry},
       title={Splitting \({S}^4\) over \(\textbf{R}{P}^2\) via the branched cover of \(\textbf{C}{P}^2\) over \({S}^4\)},
        date={1982},
     journal={Proceedings of the AMS},
      volume={\textbf{86}},
      number={2},
       pages={328\ndash 330},
}

\bib{Law2}{article}{
      author={Lawson, Terry},
       title={Detecting the standard embedding of \(\textbf{R}{P}^2\) in \({S}^4\)},
        date={1984},
     journal={Mathematische Annalen},
      volume={\textbf{267}},
       pages={439\ndash 48},
}

\bib{Massey}{article}{
      author={Massey, William},
       title={Proof of a conjecture of {W}hitney},
        date={1969},
     journal={Pacific Journal of Mathematics},
      volume={\textbf{31}},
       pages={143\ndash 156},
}

\end{biblist}
\end{bibdiv}

\end{document}